\documentclass[12pt,reqno]{amsart}

\usepackage{graphicx, fullpage,bbm}
\usepackage{amssymb, amsmath, amsthm,tikz, bm,bbm}
\usepackage{epstopdf,enumitem,esint}
\usepackage{verbatim,mathabx}
\usepackage{hyperref}
\hypersetup{hidelinks}
\usepackage[margin=1.1in]{geometry}

\newtheorem{Theorem}{Theorem}[section]
\newtheorem{Lemma}{Lemma}[section]
\newtheorem{Proposition}{Proposition}[section]

\theoremstyle{definition}
\newtheorem{Definition}{Definition}[section]

\theoremstyle{remark}
\newtheorem{Remark}{Remark}[section]

\numberwithin{equation}{section}

\renewcommand{\u}{{\bf u}}

\newcommand{\R}{{\mathbb R}}
\newcommand{\Dv}{{\rm div}}
\newcommand{\tr}{{\rm tr}}
\newcommand{\T}{{\mathbb T}}

\newcommand{\id}{\mathrm{Id}}

\newcommand{\trace}{\mathrm{tr}}

\renewcommand{\v}{{\bf v}}

\newcommand{\m}{{\bf m}}

\newcommand{\eps}{{\varepsilon}}

\author{Robin Ming Chen}
\address{Department of Mathematics, University of Pittsburgh, PA 15260, USA.}
\email{mingchen@pitt.edu}
\author{Alexis Vasseur}
\address{Department of Mathematics, University of Texas at Austin, TX 78712, USA.}
\email{vasseur@math.utexas.edu}

\author{Dehua Wang}
\address{Department of Mathematics, University of Pittsburgh, PA 15260, USA.}
\email{dhwang@pitt.edu}

\author{Cheng Yu}
\address{Department of Mathematics, University of Florida, FL 32601, USA.}
\email{chengyu@ufl.edu}

\title[Incompressible Limit of Weak Solutions]
{Universality in the Low Mach number limit via a convex integration framework}


\begin{document}

\begin{abstract}

We study the low Mach number limit of the compressible Euler equations through the lens of convex integration. For any prescribed $L^2$ weak solution of the incompressible Euler equations, we construct a corresponding family of weak solutions to the compressible Euler equations via a refined convex integration scheme. We then prove that, as the Mach number tends to zero, this family of solutions converges strongly to the given incompressible solution. This result demonstrates that the incompressible system acts as a universal attractor in this setting: every incompressible flow can be realized as the limit of convex integration solutions to the compressible system. Our approach highlights a new form of universality for singular limits and provides a rigorous framework for understanding the incompressible limit from the perspective of weak solution theory.

\end{abstract}

\maketitle 

\section{Introduction}

The incompressible Euler equations are universally regarded as the formal low Mach number limit of the compressible Euler equations. This physical premise is rooted in the observation that when fluid velocities are much smaller than the speed of sound, density variations become negligible, and the flow behaves as if it were incompressible \cite{KM2, Majda, MaS}. The mathematical theory of this singular limit aims to rigorously justify the transition from solutions of the compressible Euler system
\begin{equation} 
\label{compressible Euler}
\left\{
\begin{split}
&\partial_t \rho + \Dv (\rho \v)= 0,
\\&\partial_t (\rho \v) + \Dv (\rho \v \otimes \v) + \nabla P = 0,
\end{split}
\right.
\end{equation} 
where $\rho$ is the density, $\v$ the velocity, and $P(\rho) = a \rho^\gamma$ the pressure ($a > 0$, $\gamma > 1$), to those of the incompressible Euler system
\begin{equation}\label{incom Euler} 
\left\{
\begin{split}&
\partial_t \u + \u\cdot\nabla\u+\nabla \pi=0,
\\&  \Dv \u=0.
\end{split}\right.
\end{equation}
To capture this low Mach number regime, one introduces a scaling where time is slowed down and velocity is rescaled by a small parameter
\[
{\rho}(x,t) =: \rho_\delta(x,\delta t), \quad {\v}(x,t) =: \delta \v_\delta(x,\delta t).
\]
This leads from \eqref{compressible Euler} to the rescaled compressible Euler equations
\begin{equation}
\label{scalling system 111}
\left\{
\begin{split}
&\partial_t\rho_{\delta}+\Dv(\rho_{\delta} \v_{\delta})=0,
\\&\partial_t(\rho_{\delta} \v_{\delta})+\Dv(\rho_{\delta} \v_{\delta}\otimes \v_{\delta})+\nabla\frac{\rho_{\delta}^\gamma}{\delta^2}=0.
\end{split}
\right.
\end{equation} 
Formally, as $\delta \to 0$, we expect $\rho_\delta$ to converge to a constant (say, $1$), and the velocity field to become divergence-free, recovering \eqref{incom Euler}. 

The study of this limit was initiated by the seminal works of Ebin \cite{Ebin} and Klainerman--Majda \cite{KM, KM2}, who treated the case of smooth, well-prepared initial data. Subsequent research expanded this theory to various contexts, including ill-prepared data \cite{Ukai} , bounded domains \cite{Scho}, non-isentropic flows \cite{MS2,MS0}, and the framework of weak solutions for the Navier--Stokes equations \cite{Feireisl-incom,LM, LM2, LM3}. This extensive body of work (see also \cite{BDGL,BFH, CJ, FKMV, Fu,SEC}) typically relies on strong compactness arguments or relative entropy methods. These techniques require sufficient a priori estimates or structural assumptions on the solutions, effectively filtering out highly oscillatory or turbulent behavior.
To obtain global-in-time results, Feireisl--Klingenberg--Markfelder \cite{FKM} introduced the framework of dissipative measure-valued (DMV) solutions, establishing convergence to the smooth incompressible solution for finite-energy data. Their result demonstrates that under the DMV framework, the incompressible limit acts as a selective filter for families of approximate solutions. 

In contrast, our work asks if the incompressible system is a \emph{universal attractor}: can every incompressible weak solution be targeted and approximated by a sequence of weak solutions to the compressible system?

\subsection*{A New Perspective from Convex Integration}
This question is motivated by the modern theory of weak solutions offered by \emph{convex integration}. In a marked departure from prior approaches, this framework has revealed that for a dense set of initial data, the incompressible and compressible isentropic Euler system admits infinitely many weak solutions, even when imposing the energy inequality as an admissibility criterion \cite{CVY,SW}. The main achievement in our prior work \cite{CVY} is the development of a convex integration scheme capable of handling a general positive definite Reynolds stress tensor, a  generalization beyond earlier methods.

We now re-examine the foundational low Mach number limit through this lens. The discovery of ubiquitous non-uniqueness prompts us to shift the analytical paradigm: instead of investigating whether a given family of compressible solutions converges, we constructively address whether any arbitrary incompressible weak solution $\u$ can be realized as the limit (as $\delta \to 0$) of a sequence of compressible weak solutions built via convex integration. An affirmative answer would demonstrate a form of universality, thereby establishing that the incompressible system is indeed a universal attractor within this framework, independent of the small-scale turbulent microstructure permitted by the convex integration machinery.

\subsection*{Strategy and the Key Innovation}
Our strategy builds directly on the framework of \cite{CVY} and proceeds in three conceptual stages. First, from a given incompressible solution $\u$, we construct a smooth incompressible subsolution $(\u^\eps, \pi^\eps, R^\eps)$ to a relaxed system, with a positive definite Reynolds stress $R^\eps > 0$. Second, we design a corresponding compressible subsolution $(\rho_\delta, V_\delta, R_\delta)$ to \eqref{scalling system 111} that approximates the incompressible one, ensuring $\rho_\delta \sim 1$ and preserving the crucial property $R_\delta > 0$. Third, we apply the convex integration scheme of \cite{CVY} to this subsolution to generate infinitely many true weak solutions $(\rho_\delta, \widehat{V}_\delta)$ of \eqref{scalling system 111}. 

The central new challenge, absent in \cite{CVY}, is the requirement of uniform approximation: the constructed compressible solutions must converge strongly to $\u$ as $\delta \to 0$. This demand for controlled asymptotics forces a critical refinement of the convex integration technique. Our principal innovation is the introduction of a new $L^2$-type constraint into the non-constructive $L^\infty$ convex integration scheme. We define the set of admissible perturbations $X_0$ to include only those $\widetilde{V}$ satisfying
\begin{equation}\label{eq new constr}
\left| \int \frac{V_\delta \cdot \widetilde{V}}{\rho_\delta} \,dxdt \right| \le \frac18 \int \frac{|\widetilde{V}|^2}{\rho_\delta}\,dxdt,
\end{equation}
where $(\rho_\delta, V_\delta, R_\delta)$ is the underlying subsolution. This constraint is the key that unlocks a uniform energy estimate. It enables us to prove that the kinetic energy of the perturbations is controlled by the trace of the Reynolds stress:
\begin{equation}\label{eq unif}
\int \frac{|\widetilde{V}|^2}{\rho_\delta}\,dxdt \le C \int \trace R_\delta \,dxdt.
\end{equation}
Since our subsolution construction ensures that $\trace R_\delta$ remains $O(\delta)$, this estimate 
tames the turbulent oscillations uniformly with respect to the Mach number, ensuring their vanishing in the limit and thereby securing the desired strong convergence.

Such a control of the {\it turbulent} perturbations by the Reynolds tensor comes naturally when considering the $C^\alpha$ theory of convex integration (the constructivist theory). It was already used, in this context,  to obtain the inviscid limit of the incompressible Navier--Stokes equation for instance (see \cite{Annals}). However, to our knowledge, it was not yet derived for the $L^\infty$ theory of convex integration (using the Bair\'e category theorem). 

\subsection*{Main Result}
Employing this convex integration strategy, we prove that the incompressible system is indeed a universal attractor within this framework. For the precise notions of weak solutions to the incompressible Euler system \eqref{incom Euler} and the rescaled compressible Euler system \eqref{scalling system 111}, see the definitions below.
\begin{Definition}
\label{definition of weak solution incom}
We say $\u \in L^{\infty}(0,T;L^2(\mathbb{T}^n)) $ is a weak solution of \eqref{incom Euler} with intial data $\u^0 \in L^2$ if it is divergence-free in the sense of distribution and 
\begin{itemize}
\item for any $\phi \in C^\infty_c(\mathbb{T}^n \times [0, T); \R^n)$ with $\Dv \, \phi = 0$,
\begin{equation}\label{weak incomp euler eqn}
\int^\infty_0 \int_{\mathbb{T}^n} \u \cdot \partial_t \phi + \u \otimes \u : \nabla \phi  \,dxdt =  -\int_{\mathbb{T}^n} \u^0 \cdot \phi(\cdot, 0) \,dx.
\end{equation}
\end{itemize}
\end{Definition}

\begin{Definition}
\label{definition of weak solutions compressible} 
For $\gamma>1$, we say $(\rho_{\delta}, \rho_{\delta} \v_{\delta})$ is a weak solution of \eqref{scalling system 111} with initial value $(\rho^0_{\delta},\m_{\delta}^0) \in L^\gamma \times L^2$ if the following holds true.
\begin{itemize}
\item $\rho_{\delta}\in L^{\infty}(0,T;L^{\gamma}(\mathbb{T}^n)), \;\;\sqrt{\rho_{\delta}} \v_{\delta}\in L^{\infty}(0,T;L^2(\mathbb{T}^n))$;
\item  for any $\varphi \in C^\infty_c(\mathbb{T}^3 \times [0, T))$,
\[
\int_0^\infty \int_{\mathbb{T}^n} \left( \rho_{\delta} \, \partial_t \varphi + \rho_{\delta} \v_{\delta} \cdot \nabla \varphi \right) \, dx \, dt + \int_{\mathbb{T}^n} \rho_{\delta}^0(x) \varphi(0,x) \, dx = 0;
\]
 
 \item for any $\phi \in C^\infty_c(\mathbb{T}^n \times [0, T); \R^n)$,
\begin{equation*}
\begin{split}
\int^\infty_0 \int_{\mathbb{T}^n} \rho_{\delta} \v_{\delta} \cdot \partial_t \phi &+ \rho_{\delta} \v_{\delta} \otimes \v_{\delta} : \nabla \phi  \,dxdt +\int_0^\infty\int_{\mathbb{T}^n}\frac{\rho_{\delta}^{\gamma}}{\delta^2}\Dv \phi\,dx\,dt
\\&=  -\int_{\mathbb{T}^n}\m_{\delta}^0 \cdot \phi(\cdot, 0) \,dx.
\end{split}
\end{equation*}
\end{itemize}
\end{Definition}

 \smallskip
 
Our main result is stated as follows.
\begin{Theorem}
\label{main result}
Let $n = 2$ or $3$ and $\gamma > 1$. For any given $T > 0$, consider any weak solution $\u\in L^\infty(0,T;L^2(\mathbb{T}^n))$ to the incompressible Euler equation \eqref{incom Euler} with initial data $\u^0\in L^2(\mathbb{T}^n)$, as defined in Definition \ref{definition of weak solution incom}. Then for any $T' < T$, 
 there exists a small number $\delta_0>0$  such that for any Mach number $\delta<\delta_0$, there exists a family of initial data 
$$
(\rho_{\delta}^0,\m_{\delta}^0)\in C^0(\mathbb{T}^n)\times  L^2(\mathbb{T}^n)
$$
with the following properties.
\begin{itemize}
\item For any fixed $\delta>0$, there exist infinitely many weak solutions $(\rho_{\delta},\rho_{\delta}\widehat{\v}_{\delta})$ on $(0, T')$ to the rescaled compressible Euler equations \eqref{scalling system 111} emanating from the above data $(\rho_{\delta}^0,\m_{\delta}^0)$, in the sense of Definition \ref{definition of weak solutions compressible}.
\item The following asymptotic limits hold:
\begin{equation*}
\rho_{\delta}\to 1\;\;\;\text{ in } C^0([0,T']\times \mathbb{T}^n),\;\;\;\;\; \rho_{\delta}\widehat{\v}_{\delta}\to \u \;\text{ in } L^{2}(0,T';L^2(\mathbb{T}^n)),
\end{equation*}
and the initial data 
$$
\m_{\delta}^0\to \u^0\;\;\;\text{ weakly in }\;\; L^2(\mathbb{T}^n),
$$
as $\delta$ tends to zero.
\end{itemize}
\end{Theorem}
\begin{Remark}[On the adiabatic exponent $\gamma$]
The convex integration construction in Theorem \ref{main result} is valid for all $\gamma > 1$. This is in contrast to the result in \cite[Theorem 1.1]{CVY}, where the constraint $1 < \gamma \le 1 + \tfrac2n$ arises from the requirement that the constructed solutions satisfy the energy inequality. Since the present work does not impose the energy condition, our result holds in the full range $\gamma > 1$. 
\end{Remark}

This theorem establishes a rigorous density result: the set of incompressible solutions attainable as low-Mach limits of convex integration solutions is the \emph{entire} $L^2$-space. It affirms that any large-scale, divergence-free flow field can emerge as the effective, averaged description of a compressible fluid in the low Mach number regime, irrespective of the fine-scale, non-unique oscillatory dynamics modeled by convex integration. This provides a mathematical justification for the robustness of the incompressible approximation, confirming the physical intuition formalized in \cite{KM2, Majda, MaS}.

\subsection*{Outline of the Paper}
The rest of the paper is structured as follows. In Section \ref{sec sub}, we detail the construction of subsolutions, from the regularization of the incompressible solution (Lemma \ref{incom smooth}) to its lifting to a compressible subsolution for \eqref{scalling system 111} (Lemma \ref{lem subsoln}). Section \ref{sec ci} is dedicated to the convex integration scheme. We present the core discretization and $L^1$-coercivity argument under the new constraint \eqref{eq new constr} (Lemma \ref{L1-coercivity}), prove the uniform estimate \eqref{eq unif} (Proposition \ref{proposition points}), and establish the existence of infinitely many solutions via a Bair\'e category argument (Lemma \ref{lem conv int}). Finally, in Section \ref{sec pf}, we put these elements together to prove Theorem \ref{main result}, confirming the convergence in the low Mach number limit.

\bigskip

\section{Construction of subsolutions to the compressible Euler equations}\label{sec sub}
In this section, we construct a subsolution to the scaled compressible Euler equations \eqref{scalling system 111}, starting from an arbitrary weak solution of the incompressible Euler equations \eqref{incom Euler}.
This subsolution serves as the foundation for applying the convex integration framework, which in turn yields infinitely many weak solutions to \eqref{scalling system 111}. For clarity, we divide the construction into two lemmas.
\smallskip

The first lemma is stated below:
\begin{Lemma}
\label{incom smooth} 
For any weak solutions $\u$ to the incompressible Euler equations  \eqref{incom Euler} and any $\eps > 0$,
there exist $\u^{\varepsilon}, \pi^\varepsilon \in C^{\infty}([0,T-\eps] \times \mathbb{T}^n)$ and $R^\varepsilon \in C^\infty([0, T - \varepsilon] \times \mathbb{T}^n; \mathbb{S}^{n \times n})$, where $\mathbb{S}^{n \times n}$ is the set of symmetric matrices in dimension $n$,  such that
\begin{itemize}
\item
$(\u^{\varepsilon}, \pi^{\varepsilon}, R^\varepsilon)$ is a weak solution to 
\begin{equation}
\label{smooth euler}
\left\{
\begin{split}
& \partial_t \u^{\varepsilon} + \u^{\varepsilon}\cdot\nabla \u^{\varepsilon}+\nabla \pi^{\varepsilon}+\Dv R^{\varepsilon}=0,
\\& \Dv {\u^{\varepsilon}}=0.
\end{split}
\right.
\end{equation}
\item $R^{\varepsilon}>\frac{\varepsilon}{4} \id_n$ is a symmetric matrix, where $\id_n$ is the $n \times n$ identity matrix.
\item There exists a constant $C>0$ independent of $\eps$ such that
\begin{equation}\label{eq reg est}
\|\u^{\varepsilon}-\u\|_{L^2([0,T - \eps] \times \mathbb{T}^n)}+\|R^{\varepsilon}\|_{L^1([0,T - \eps] \times \mathbb{T}^n)}\leq C \varepsilon.
\end{equation}
\end{itemize}
\end{Lemma}

\begin{proof} Let $\u$ be a weak solution to the incompressible Euler equations \eqref{incom Euler}. Let $\eta : \R \times \T^n \to \R$ be a standard mollifier supported on $[-1,0] \times \T^n$. For any $\alpha > 0$ we define
\[
\eta_\alpha(t,x) := \frac{1}{\alpha^4}\eta\left( \frac{t}{\alpha}, \frac{x}{\alpha} \right),
\]
and let $\u_\alpha : [0, T - \alpha] \times \T^n$ be defined by $\u_\alpha = \eta_\alpha \ast \u$. Obviously 
\[
\u_\alpha \to \u \qquad \text{in } \ L^2([0,T-\alpha] \times\mathbb {T}^3) \ \text{ as } \ \alpha \to 0.
\]
Therefore, for every $\varepsilon>0$, there exists $\alpha_\varepsilon>0$ small enough such that 
 the mollified velocity field defined as 
 $$
 \u^{\varepsilon} := \u*\eta_{\alpha_\varepsilon}(x)
 $$
 satisfies \eqref{smooth euler} on $[0, T - \eps]$, with
 \[
 \pi^\varepsilon := \pi \ast \eta_{\alpha_\varepsilon}, \qquad R^{\varepsilon} :=(\u\otimes\u)^{\varepsilon}-\u^{\varepsilon}\otimes \u^{\varepsilon}+\frac{\varepsilon}{2} \id_n,
 \]
where $(\u\otimes\u)^{\varepsilon} : = (\u\otimes\u) \ast \eta_{\alpha_\varepsilon}$.
 
Clearly by the convexity of the nonlinearity we know that the Reynolds-type stress tensor $ R^{\varepsilon}$ is a symmetric and positive definite matrix with 
\[
R^{\varepsilon} \ge \frac{\varepsilon}{2} \id_n > \frac{\varepsilon}{4} \id_n.
\]
Moreover we have
 $$
 \|\u^{\varepsilon}-\u\|_{L^2([0,T - \eps] \times \mathbb{T}^n)}\leq C \varepsilon,\qquad 
 \|(\u\otimes \u)^{\varepsilon}-\u^\eps \otimes \u^\eps\|_{L^1([0,T -\eps] \times \mathbb{T}^n)} \leq C \varepsilon,
 $$
 and 
$$
\|R^{\varepsilon}\|_{L^1([0,T - \eps] \times \mathbb{T}^n)}\leq C \varepsilon,
$$
which completes the proof of the lemma. 
\end{proof}

\smallskip

Once this lemma is established, we are ready to construct a subsolution to the compressible Euler equations \eqref{scalling system 111}. For $\varepsilon, \delta, K > 0$, define
\begin{subequations}\label{eq varrho}
\begin{equation}\label{eq def varrho}
\varrho_{\delta,K}(t,x) := \frac{(\delta^2 \pi^\varepsilon + K)^{1/\gamma} - 1}{\delta^2},
\end{equation}
where $\pi^\varepsilon$ is from Lemma \ref{incom smooth}. Since $\pi^\varepsilon \in C^\infty([0, T-\varepsilon] \times \T^n)$, we know that for any $\delta > 0$ there exists a sufficiently large $K > 0$ such that $\delta^2 \pi^\varepsilon + K > 0$, and hence $\varrho_{\delta, K} \in C^\infty([0, T-\varepsilon] \times \T^n)$. Considering $\delta$ sufficiently small such that $\| \delta^2 \pi^\eps \|_{L^\infty} \ll1$, among those choices of $K$, a simple intermediate value theorem further implies that for any $t \in [0, T-\varepsilon]$ there exists some $K_\ast = K_\ast(\delta, t) > 0$ such that
\begin{equation}\label{eq ave varrho}
\left\langle \varrho_{\delta, K_\ast} \right\rangle := \fint_{\T^n} \varrho_{\delta, K_\ast} \,dx = 0. 
\end{equation}
\end{subequations}
\begin{Lemma}\label{lem subsoln}
\label{subsolution}
For any $\varepsilon>0$, consider $(\u^\varepsilon, \pi^\varepsilon, R^\varepsilon)$ as in Lemma \ref{incom smooth}. There exists a $\delta_0 = \delta_0(\eps) >0$ such that for any $\delta\leq \delta_0$, there exist $\v_\delta \in C^{\infty}([0,T-\eps]\times \mathbb{T}^n)$ and $\tilde{R}_{\delta} \in C^\infty([0, T - \varepsilon] \times \mathbb{T}^n; \mathbb{S}^{n \times n})$ with $\tilde{R}_{\delta} > 0$, solving 
\begin{equation}\label{eq subsoln}
\left\{
\begin{split}
&\partial_t\rho_{\delta}+\Dv(\rho_{\delta} \v_{\delta})=0,
\\&\partial_t(\rho_{\delta} \v_{\delta})+\Dv(\rho_{\delta} \v_{\delta}\otimes \v_{\delta})+\nabla\frac{\rho_{\delta}^\gamma}{\delta^2}+\Dv(R^{\varepsilon}+\tilde{R}_{\delta})=0,
\end{split}\right.
\end{equation}
where 
\begin{equation}\label{eq def rho}
\rho_\delta(t,x) := 1 + \delta^2 \varrho_{\delta, K_\ast}(t,x) 
\end{equation}
with $\varrho_{\delta, K_\ast}$ being given in \eqref{eq varrho}. Moreover we have
\begin{equation}\label{eq subsoln est}
\begin{split}
\|\rho_{\delta}-1\|_{C^0([0,T-\varepsilon] \times\mathbb{T}^n)} + \|\rho_{\delta}\v_{\delta}-\u^{\varepsilon}\|_{L^2([0,T-\varepsilon] \times \mathbb{T}^n)} + \|\tilde{R_{\delta}}\|_{L^\infty([0,T-\varepsilon] \times \mathbb{T}^n)}\leq \varepsilon.
\end{split}
\end{equation}
\end{Lemma}
\begin{proof}
From \eqref{eq varrho} and \eqref{eq def rho} we have 
\[
\rho_\delta = 1 + \delta^2 \varrho_{\delta, K_\ast} = (\delta^2 \pi^\varepsilon + K_\ast)^{1/\gamma}.
\]
So for $\delta$ sufficiently small we know that $\rho_\delta > 0$, and 
\begin{equation}\label{eq prop rho}
\rho_\delta \in C^\infty([0, T-\varepsilon] \times \T^n), \quad \left\langle \rho_\delta \right\rangle = 1, \quad \text{and} \quad \|\rho_{\delta} - K_\ast^{1/\gamma}\|_{C^0([0,T-\varepsilon] \times\mathbb{T}^n)} \le C(\varepsilon) \delta^2.
\end{equation}
The time derivative of $ \rho_{\delta} $ is given by
\[
\partial_t \rho_{\delta} = \frac{1}{\gamma} \left(\delta^2 \pi^{\varepsilon} + K_\ast \right)^{\frac{1}{\gamma} - 1} (\delta^2 \partial_t \pi^{\varepsilon} + \partial_t K_\ast).
\]

It is easily seen that $K_\ast$ is smooth in $(\delta, t) \in [0, \infty) \times [0, T - \varepsilon]$, and $K_\ast(0, t) = 1$. Therefore we have for $\delta$ small enough
\[
\| K_\ast(\delta, t) - 1 \|_{C^0([0, T - \varepsilon])} \le C(\varepsilon) \delta, \qquad \| \partial_t K_\ast(\delta, t) \|_{C^0([0, T - \varepsilon])} \le C(\varepsilon) \delta.
\]
This leads to the following bound:
\begin{equation}
\label{bound of density in delta}
\|\rho_{\delta} - 1\|_{C^0([0,T-\varepsilon] \times\mathbb{T}^n)} + \| \partial_t \rho_{\delta} \|_{L^\infty([0,T-\varepsilon] \times\mathbb{T}^n)} + \| \partial_t^2 \rho_{\delta} \|_{L^\infty([0,T-\varepsilon] \times\mathbb{T}^n)} \leq C(\varepsilon) \delta.
\end{equation} 

With these bounds in place, we may now proceed with the following calculations. Define $\m_\delta$ on $[0, T - \varepsilon] \times \mathbb{T}^n$ through
\[
\partial_t \rho_\delta + \Dv \m_\delta = 0.
\]
From \eqref{eq prop rho} we see that $\left\langle \partial_t \rho_\delta \right\rangle = 0$, and hence by the well-known result of Bourgain--Brezis \cite{BB} it follows that the above equation admits a solution $\m_\delta = \nabla (-\Delta)^{-1} (\partial_t \rho_\delta)$, with the estimate 
\begin{equation}\label{eq est m}
\|\m_\delta\|_{L^\infty([0,T-\varepsilon] \times\mathbb{T}^n)} \leq C \| \partial_t \rho_{\delta} \|_{L^\infty(0,T-\varepsilon; L^n(\mathbb{T}^n))} \le C(\varepsilon) \delta.
\end{equation}
Here $(-\Delta)^{-1}$ is taken to to be an operator that sends mean-zero $L^2$ functions to mean-zero $H^2$ functions. 
From this, and the fact that $\rho_\delta > 0$, we define
\[
\v_\delta := \frac{\u^\epsilon + \m_\delta}{\rho_\delta}, \qquad \text{i.e.} \qquad \rho_\delta \v_\delta = \u^\epsilon + \m_\delta.
\]
From \eqref{eq est m} it follows that
\begin{equation}
\label{m close}
\|\rho_{\delta}\v_{\delta}-\u^{\varepsilon}\|_{L^2([0,T-\varepsilon] \times\mathbb{T}^n)}=\| \m_{\delta}\|_{L^2([0,T-\varepsilon] \times\mathbb{T}^n)} \leq C(\eps) \delta.
\end{equation}

Finally we find $\tilde{R}_{\delta}$. For this, we calculate the following equation
\begin{equation*}
\begin{split}
&\partial_t(\u^{\varepsilon} + \m_{\delta}) + \Dv \left( \frac{(\u^{\varepsilon}+ \m_{\delta})\otimes (\u^{\varepsilon}+ \m_{\delta})} {\rho_{\delta}} \right)+\nabla\frac{\rho_{\delta}^{\gamma}}{\delta^2}+\Dv R^{\varepsilon}
\\&=\partial_t \m_{\delta}+ \Dv \left( \frac{(\u^{\varepsilon} + \m_{\delta})\otimes (\u^{\varepsilon}+ \m_{\delta})} {\rho_{\delta}}-\u^{\varepsilon}\otimes\u^{\varepsilon} \right)
\\&= \Dv \left( (-\Delta)^{-1}\partial^2_{t}\rho_{\delta} \id_n \right)+\Dv \left( \frac{(\u^{\varepsilon}+ \m_{\delta})\otimes (\u^{\varepsilon}+\m_{\delta})} {\rho_{\delta}}-\u^{\varepsilon}\otimes\u^{\varepsilon} \right).
\end{split}
\end{equation*}
Since $\partial_t^2 \rho_\delta(t, \cdot) \in L^2(\T^n)$ with $\left\langle \partial_t^2 \rho_\delta \right\rangle = 0$, $(-\Delta)^{-1}\partial^2_{t}\rho_{\delta}$ is well-defined. 
Thus, we denote  
$$
\tilde{R}_{\delta} := -(-\Delta)^{-1}\partial^2_{t}\rho_{\delta} \id_n - \frac{(\u^{\varepsilon}+ \m_{\delta})\otimes (\u^{\varepsilon}+\m_{\delta})} {\rho_{\delta}}+\u^{\varepsilon}\otimes\u^{\varepsilon}+\frac{\varepsilon}{8} \id_n.
$$
Then we obtain the equation
\begin{equation*}
\partial_t(\u^{\varepsilon}+ \m_{\delta}) + \Dv \left( \frac{(\u^{\varepsilon}+ \m_{\delta})\otimes (\u^{\varepsilon}+ \m_{\delta})} {\rho_{\delta}} \right)+\nabla\frac{\rho_{\delta}^{\gamma}}{\delta^2}+\Dv (R^{\varepsilon}+\tilde{R}_{\delta})=0.
\end{equation*}
From the above estimates, we obtain 
\[
\left\| \frac{(\u^{\varepsilon}+ \m_{\delta})\otimes (\u^{\varepsilon}+\m_{\delta})} {\rho_{\delta}} - \u^{\varepsilon}\otimes\u^{\varepsilon} \right\|_{L^\infty([0,T-\varepsilon] \times \mathbb{T}^n))} \le C(\varepsilon) \delta.
\]
Since $\left\langle \partial_t^2 \rho_\delta \right\rangle = 0$, the $L^2$-estimate on torus for the first term in $\tilde{R}_{\delta}$ gives
\[
\left\| (-\Delta)^{-1} \partial_t^2 \rho_\delta \right\|_{L^\infty(0,T-\varepsilon; H^2(\mathbb{T}^n))} \le C \left\| \partial_t^2 \rho_\delta \right\|_{L^\infty(0,T-\varepsilon; L^2(\mathbb{T}^n))} \le C(\varepsilon) \delta,
\]
which, by Sobolev embedding (since here $n = 2$ or $3$), yields
\[
\left\| (-\Delta)^{-1} \partial_t^2 \rho_\delta \right\|_{L^\infty([0,T-\varepsilon] \times \mathbb{T}^n)} \le C(\varepsilon) \delta.
\]
Therefore
\begin{equation}\label{R close}
\| \tilde{R}_{\delta} \|_{L^\infty([0,T-\varepsilon] \times \mathbb{T}^n))} \leq C(\varepsilon) \delta.
\end{equation}

Choosing $\delta$ sufficiently small (depending on $\eps$) it follows that $\tilde{R}_{\delta} > 0$, and \eqref{eq subsoln est} follows from \eqref{bound of density in delta}, \eqref{m close} and \eqref{R close}, and hence we complete the proof of the lemma. 
\end{proof}

We call the solution $(\rho_\delta, V_\delta := \rho_\delta \v_\delta, R_\delta := R^\varepsilon + \tilde{R}_\delta)$ to \eqref{eq subsoln} a subsolution to the isentropic compressible Euler system \eqref{scalling system 111}.  

\bigskip

\section{Convex integration}\label{sec ci}

We begin with a (smooth) subsolution $(\rho_0, V_0, R_0)$ to the isentropic compressible Euler system, as constructed in Lemma \ref{lem subsoln}, which solves the following system:
\begin{equation}
\left\{
\begin{split}
\label{isentropic equivalent}
&\partial_t\rho + \Dv V = 0, \\
&\partial_t V + \Dv U + \nabla\left(p_\delta(\rho) + \frac{|V|^2}{n\rho}\right) + \Dv R = 0,
\end{split}
\right.
\end{equation}
where $ U := \frac{V \otimes V}{\rho} - \frac{|V|^2}{n\rho} \id_n $, and $ R > 0 $ is a symmetric matrix-valued function representing the Reynolds stress. Here the pressure $p_\delta(\rho) := \rho^\gamma/\delta^2$. Our objective is to construct infinitely many bounded perturbations $ (\widetilde{V}, \widetilde{U}) $, supported in a prescribed domain $ P \subset (0,T) \times \mathbb{T}^n $, that satisfy the following linear system:
\begin{equation}
\label{ci system}
 \left\{
\begin{split}
&\Dv \widetilde{V}=0,\\
&\partial_t\widetilde{V}+\Dv \widetilde{U}=0,
\end{split}\right.
\end{equation} 
together with the nonlinear algebraic constraint:
\begin{equation}
\label{restriction =}
\frac{(V_0 + \widetilde{V}) \otimes (V_0 + \widetilde{V})}{\rho_0} - (U_0 + \widetilde{U}) = \frac{|V_0|^2}{n\rho_0} \id_n + R_0 \quad \text{a.e. in } P,
\end{equation}
where
\begin{equation}
\label{U0 condition}
U_0 := \frac{V_0 \otimes V_0}{\rho_0} - \frac{|V_0|^2}{n\rho_0} \id_n.
\end{equation}

This construction is made possible by the convex integration framework. The existence of such perturbations implies that $(\rho_0, V_0 + \widetilde{V})$ satisfies 
\[
\left\{\begin{split}
& \partial_t \rho_0 + \Dv (V_0 + \widetilde{V}) = 0, \\
& \partial_t (V_0 + \widetilde{V}) + \Dv \left( \frac{(V_0 + \widetilde{V}) \otimes (V_0 + \widetilde{V})}{\rho_0} \right) + \nabla p_\delta(\rho_0) = 0.
\end{split}\right.
\]
Setting $\widehat{\v}_\delta : = \tfrac{V_0 + \widetilde{V}}{\rho_0}$ leads to the weak solution $(\rho_0, \widehat{\v}_\delta)$ 
to the compressible Euler system \eqref{scalling system 111}.

In our application, $(\rho_0, V_0, R_0)$ will be chosen from the subsolutions as in Lemma \ref{lem subsoln}. In particular, $\rho_0 = \rho_\delta$, $V_0 = \rho_\delta \v_\delta$, $R_0 = R^{\varepsilon}+\tilde{R}_{\delta}$, and we know that $\rho_0 > 0$. Therefore, on $P$ we impose the uniform condition
\begin{equation}
 \label{bounds rho}
 0<\frac{1}{\Lambda^2}\leq \rho_0\leq \Lambda^2.
 \end{equation}
Note that the system \eqref{ci system}--\eqref{U0 condition} does not involve the pressure function $p_\delta$, and therefore all of the relevant estimates in this section are uniform in $\delta$. Abusing notations, we will still use $\varepsilon$ and $\delta$ in the following, but they are independent of the previous section. 

In contrast to the approach in \cite{CVY}, our construction of infinitely many solutions $ (\widetilde{V}, \widetilde{U}) $ to the system \eqref{ci system}--\eqref{restriction =} includes a new and essential feature. Specifically, we will ask that the following  inequality holds:
\begin{equation}
\label{additional control}
\left|\int_0^T\int_{\mathbb{T}^3}\frac{V_0\cdot\widetilde{V}}{\rho_0}\,dx\,dt\right|\leq \frac{1}{8}\int_0^T\int_{\mathbb{T}^3} \frac{|\widetilde{V}|^2}{\rho_0}\,dx\,dt.
\end{equation}
This inequality is not present in \cite{CVY}, and it constitutes a new estimate of the convex integration framework. It will allow one to control the turbulent part $\widetilde V$ by the Reynolds tensor $R_0$ (see Proposition \ref{proposition points}).

\medskip

Following \cite{CDK,CVY,DS}, we will consider the relaxed condition
\begin{equation}
\label{ci constraint 2}
\frac{(V_0 + \widetilde{V}) \otimes (V_0 + \widetilde{V})} { \rho} - (U_0 + \widetilde U) < {|V_0|^2 \over n\rho_0} \id_n + R_0,
\end{equation}
and define the set 
\begin{equation}\label{defn X_0}
X_0 := \left\{ (\widetilde V, \widetilde U) \in C^\infty_c(P; \R^n \times \mathbb{S}^{n\times n}_0): \ (\widetilde V, \widetilde U) \text{ solves } \eqref{ci system},   \eqref{additional control} \text{ and } \eqref{ci constraint 2} \right\},
\end{equation}
where $\mathbb{S}^{n \times n}_0$ is the set of symmetric traceless matrices in dimension $n$. 

It is clear that the set $ X_0 $ is nonempty, since $ 0 \in X_0 $ due to the condition \eqref{U0 condition} and the assumption $ R_0 > 0 $. We then define $ X $ to be the closure of $ X_0 $ in the weak-$\ast$ topology of $ L^\infty $. The boundedness of $ X $ in $ L^\infty $ ensures that this weak-$\ast$ topology is metrizable on $ X $ with a well-defined metric $d$, which in turn implies that $ (X, d) $ is a complete metric space.

The main idea is similar to that of \cite{CVY}. That is, we would like to show that the saturation condition \eqref{restriction =} holds on a residual set of $X$, and thus a Bair\'e category argument can be applied to yield infinitely many solutions to \eqref{ci system}--\eqref{restriction =}. However, here the difference lies, as pointed out above, in the extra condition \eqref{additional control}. The key step in \cite{CVY} is to establish an $L^1$-coercivity property for the weak-$\ast$ convergence in terms of the trace of the ``defect matrix''
$$
M := {|V_0|^2 \over n\rho_0} \id_n + R_0-\frac{(V_0 + \widetilde{V}) \otimes (V_0 + \widetilde{V})} { \rho_0} + (U_0 + \widetilde U)
$$
with 
\[
\lambda_{\min}(M) := \text{ the smallest eigenvalue of } M.
\]
The goal is to reproduce such an $L^1$-coercivity, with the additional constraint \eqref{additional control}. 

\begin{Lemma}[$L^1$-coercivity]
\label{L1-coercivity}
For any $(\widetilde{V}, \widetilde{U}) \in X_0$, where $X_0$ is defined in \eqref{defn X_0}, there exists a sequence $\{(\widetilde{V}_i, \widetilde{U}_i)\} \subset X_0$ that converges weakly-$*$ to $(\widetilde{V}, \widetilde{U})$, such that
$$
\|\widetilde{V}_i - \widetilde{V}\|_{L^1(P)} \geq \frac{c_0}{\Lambda} \int_{P} \trace \, M \, dx \, dt,
$$
for some positive geometric constant $c_0$, where $\Lambda$ is given in \eqref{bounds rho}.
\end{Lemma}

\begin{proof}

The proof proceeds via discretization and is divided into three steps. The first and third steps are the same as in \cite[Lemma 3.2]{CVY}, but for the reader's convenience, we include the full details.

{\bf Step 1. }
We begin by considering a localized problem. From \eqref{ci constraint 2} we know that $M > 0$ on $P$. It follows from the smoothness of $M$ that there exist a compact subsets $\Omega_1 \subset \Omega \subset P$, and a sufficiently small $\delta_1 > 0$ such that any cube with center in $\Omega_1$ and size $\delta_1$ is inside $\Omega$, and
\begin{align*}
& \int_{\Omega_1} \trace M \,dxdt \ge \frac12 \int_{P} \trace M \,dxdt, \\
& \lambda_\ast := \min_{\Omega} \left\{ \min( \lambda_{\min}(M), \lambda_{\min}(R_0) ) \right\} > 0. 
\end{align*}

Let $ (x_0, t_0) \in \Omega_1 $ be a fixed point, and choose a small open cube $ Q $ centered at $ (x_0, t_0) $ with size is smaller than $ \delta_1 $. Denote the following averages over $ Q $:
\begin{equation}\label{averages}
\begin{split}
\overline V := \fint_Q (V_0+\widetilde V)\,dxdt, \quad \overline U := \fint_Q (U_0+\widetilde U)\,dxdt, \quad \overline R_0 :=   \fint_Q R_0\,dxdt.
\end{split}
\end{equation}

Let
\begin{equation}\label{def_C_Q}
\underline{\rho} := \min_Q \rho_0, \quad 
C_Q := \min_Q \frac{|V_0|^2}{\rho_0} \geq 0, \quad 
\overline{R_Q} := -\frac{\lambda_*}{16n} \, \id_n + \overline{R_0},
\end{equation}
and define 
\begin{equation}\label{def_M_Q}
\overline{M_Q} := \frac{C_Q}{n} \, \id_n + \overline{R_Q} - \frac{\overline{V} \otimes \overline{V}}{\underline{\rho}} + \overline{U} \in \mathbb{S}^{n \times n}.
\end{equation}

Note that the uniform continuity of $ \rho_0 $, $ V_0 $, $ R_0 $, and $ (\widetilde{V}, \widetilde{U}) $ in $ \Omega $ implies that for any $ \varepsilon > 0 $, there exists some $ \delta > 0 $ independent of $ Q $ such that whenever $ |Q| < \delta $, the fluctuation of these quantities over $ Q $ is smaller than $ \varepsilon $. In particular, by choosing $ \varepsilon $ small enough with respect to $ \lambda_* $, we can ensure that for $ \delta $ sufficiently small,
\begin{equation*}
\begin{split}
\overline{R_Q} > 0, &\quad 
\sup_Q \| M - \overline{M_Q} \| < \frac{\lambda_*}{8n}, \quad
\sup_Q \| \overline{R_0} - R_0 \| < \frac{\lambda_*}{64n}, \quad 
\sup_Q \left| C_Q - \frac{|V_0|^2}{\rho_0} \right| < \frac{\lambda_*}{64n},
\end{split}
\end{equation*}
where $ \|\cdot\| $ denotes the standard matrix norm.

Together with \eqref{def_C_Q} we obtain the following bound:
\begin{equation}\label{bound M_Q}
\frac{|V_0|^2}{n \rho_0} \id_n + \overline R_Q < \frac{\lambda_*}{32n} \id_n + R_0,
\end{equation}
and 
\begin{equation}\label{M bound}
\overline{M_Q} = M + (\overline{M_Q} - M) > M - \frac{\lambda_*}{8n} \id_n > 0, \text{ and thus} \quad \trace  \overline{M_Q} > \frac{1}{4} \trace  M \text{ on } Q.
\end{equation}
Consider the rescaled set $\mathring Q := \{ (x, t/\sqrt{\underline\rho }): \ (x, t) \in Q \}$ and define 
\[
\begin{split}
X_0^Q := & \Big\{ (V, U) \in C^\infty_c(\mathring{Q}; \R^n \times \mathbb{S}_0): \ (V, U) \text{ solves } \eqref{ci system} \text{ and }  \\
&\qquad \qquad \left. (V_0 + V) \otimes (V_0 + V) - (U_0 + U) < {C_Q \over n} \id + R_0 \right\}.
\end{split} 
\]
Thanks to \eqref{def_M_Q} and \eqref{M bound}, $(0,0) \in X_0^Q$.

Since we have presented a localization argument above, we would like to establish $L^1$-coercivity of the perturbation with respect to constant states. Such a result has already been established in \cite[Proposition 2.1]{CVY}. Therefore, quoting \cite[Proposition 2.1]{CVY} with the constants 
\[
(V_0, U_0, R_0, C_0) = \left( \frac{\overline V}{\sqrt{\underline\rho }}, \overline U, \overline R_Q, C_Q \right),
\]
it follows that there exists a sequence $(\mathring V_i,\mathring U_i)\in X^Q_0$ converging weakly to 0, and such that for every $i$:
\begin{equation*}
\|\mathring V_i\|_{L^1(\mathring Q)}\geq c_1 (C_Q+\trace   \overline R_Q-|\overline V|^2)|\mathring Q|\geq c_1 (\trace   \overline{M_Q})|\mathring Q|
\end{equation*}
for some geometric constant $c_1 > 0$. Since $(\mathring V_i, \mathring U_i)\in X^Q_0$, it verifies   \eqref{ci system} and 
$$
\left(\frac{\overline V}{\sqrt{\underline\rho }}+\mathring V_i \right)\otimes \left(\frac{\overline V}{\sqrt{\underline\rho }}+\mathring V_i \right)-\overline  U+\mathring U_i< \frac{C_Q}{n}\id_n +  \overline R_Q.
$$

Consider the change of variable $(\mathring{V_i}, \mathring{U_i})(x,t) := \left( \frac{V_i}{\sqrt{\underline{\rho}}}, U_i \right) (x, t\sqrt{\underline{\rho}})$. The functions $(V_i, U_i)$ are now compactly supported in $Q$, and they still verify \eqref{ci system} and converge weakly to 0. Unraveling the change-of-variables and using \eqref{M bound} we find that 
\begin{equation}\label{punch3}
\| V_i \|_{L^1(Q)} = \underline{\rho} \|\mathring V_i\|_{L^1(\mathring Q)} \ge c_1 \underline{\rho} (\trace   \overline{M_Q})|\mathring Q| = c_1 \sqrt{\underline{\rho}} (\trace   \overline{M_Q})|Q| \geq \frac{c_1}{4\Lambda} \int_Q \trace  \, M \, dx dt.
\end{equation}
Moreover, we have the following estimates on $Q$ 
\begin{equation}\label{condition1 of V+Vi}
\begin{split}
& \frac{(V_0+\widetilde{V}+V_i) \otimes (V_0+\widetilde{V}+V_i)}{\rho_0} - (U_0+\widetilde{U}+U_i) \\
& \quad \leq  \frac{(V_0+\widetilde{V}+V_i) \otimes (V_0+\widetilde{V}+V_i)}{\underline{\rho}} - (U_0+\widetilde{U}+U_i) \\
& \quad \leq  \frac{(\overline{V} +V_i) \otimes (\overline{V} +V_i)}{\underline{\rho}} - (\overline{U}+U_i) + C(|V_0+\widetilde{V}-\overline{V_0}| + |U_0+\widetilde{U}-\overline{U_0}|) \, \id_n \\
& \quad \leq  \frac{{C_Q}}{n} \, \id_n + \overline{R_Q} + \frac{\lambda_*}{64n} \, \id_n \\
& \quad < \frac{|V_0|^2}{n\rho_0} \, \id_n + R_0,
\end{split}
\end{equation}
where we used the definition of $\underline{\rho}$ for the first inequality, and \cite[Remark 2.2]{CVY} which ensures that the constant $C$ is independent of the sequence $V_i$ for the second one. With the constant $C$ fixed, we obtain the third inequality by taking $\delta$ even smaller if necessary. The last inequality follows from \eqref{def_C_Q}.

\textbf{Step 2.} In this step, we verify that the function $ \widetilde{V} + V_i $ satisfies the inequality \eqref{additional control}. This verification is a key part of the argument, as it ensures that $ \widetilde{V} + V_i $ meets the necessary condition to belong to the set $ X_0 $.

We let $$\overline{V_0}:= \fint_Q V_0\,dx\;\;\text{ and }\;\;\overline{\rho_0}:=\fint_Q \rho_0\,dx.$$
Choose  a small $Q$ such that 
\begin{equation*}
\begin{split}
&\left| \frac{V_0}{\rho_0}-\frac{\overline{V_0}}{\overline{\rho_0}} \right| \leq \left| \frac{V_0-\overline{V_0}}{\rho_0} \right| + \left| \frac{\overline{V_0}}{\rho_0}-\frac{\overline{V_0}}{\overline{\rho_0}} \right| \le \frac{1}{\Lambda^2} |V_0-\overline{V_0}| + \frac{\overline{V_0}}{\Lambda^4} |\rho_0-\overline{\rho_0}|
\leq \varepsilon 
\end{split}
\end{equation*}
for some $\varepsilon$ to be determined later. 
Note that $\int_{Q}V_i\,dx=0$, and from \eqref{punch3},
\[
\|V_i\|_{L^2(Q)} \ge \frac{1}{|Q|^{1/2}} \|V_i\|_{L^1(Q)} = \frac{{\underline{\rho}}}{|Q|^{1/2}} \|\mathring V_i\|_{L^1(\mathring Q)}\geq c_1 \sqrt{\underline{\rho}} (\trace   \overline{M_Q})|Q|^{1/2}.
\]
Together with \eqref{M bound}, this allows us to calculate that 
\begin{align*}
\left| \int_Q \frac{V_0 \cdot V_i}{\rho_0} \,dx \right|
&= \left| \int_Q \left( \frac{V_0}{\rho_0} - \frac{\overline{V_0}}{\overline{\rho_0}} \right) V_i\,dx \right| \leq \varepsilon \int_Q| V_i|\,dx \\
& \le \varepsilon |Q|^{1/2} \|V_i\|_{L^2(Q)} \le \frac{\varepsilon }{c_1 \sqrt{\underline{\rho}} (\trace   \overline{M_Q})} \|V_i\|^2_{L^2(Q)}  \\
& \le \frac{\varepsilon \Lambda^2 }{c_1 (\trace   \overline{M_Q})} \int_Q \frac{|V_i|^2}{\rho_0} \;dx \le \frac{4 \varepsilon \Lambda^2 }{c_1 (\trace M)} \int_Q \frac{|V_i|^2}{\rho_0} \;dx \\
& \le \frac{4 \varepsilon \Lambda^2 }{c_1 n \lambda_\ast} \int_Q \frac{|V_i|^2}{\rho_0} \;dx.
\end{align*}
Choosing $\varepsilon$ such that $\tfrac{4\varepsilon \Lambda^2}{n c_1 \lambda_\ast} \le \tfrac18$ we see that
\[
\left| \int_Q \frac{V_0 \cdot V_i}{\rho_0} \,dx \right| \le \frac18 \int_Q \frac{|V_i|^2}{\rho_0} \;dx.
\]

Assuming that
\[
\left| \int_Q \frac{V_0 \cdot \widetilde{V}}{\rho_0} \, dx \right| \leq \frac{1}{8} \int_Q \frac{|\widetilde{V}|^2}{\rho_0} \, dx,
\]
we have
\begin{align*}
\left| \int_Q \frac{V_0 \cdot (\widetilde{V} + V_i)}{\rho_0}\,dx \right| & \leq \left| \int_Q\frac{V_0 \cdot \widetilde{V}}{\rho_0} \right| + \left| \int_Q \frac{V_0 \cdot V_i}{\rho_0} \,dx\right| \\
& \leq \frac{1}{8} \int_Q \frac{|\widetilde{V}|^2}{\rho_0}\,dx + \left| \int_Q \frac{V_0 \cdot V_i}{\rho_0} \,dx\right| \\
& \leq \frac{1}{8} \int_Q \frac{|\widetilde{V}|^2}{\rho_0} \,dx+ \frac{1}{8} \int_Q \frac{|V_i|^2}{\rho_0}\,dx \\
& = \frac{1}{8} \int_Q \frac{|\widetilde{V} + V_i|^2}{\rho_0}\,dx - \frac{1}{4} \int_Q \frac{V_0 \cdot V_i}{\rho_0}\,dx.
\end{align*}
To ensure 
\[ \left| \int_Q \frac{V_0 \cdot (\widetilde{V} + V_i)}{\rho_0} \,dx\right|
\leq \frac{1}{8} \int_Q \frac{|\widetilde{V} + V_i|^2}{\rho_0}\,dx,
\]
we require that
\[
- \frac{1}{4} \int \frac{V_0 \cdot V_i}{\rho_0} \,dx\leq 0.
\]
One can guarantee this condition in the following way. In particular, 
if 
$$
\int_Q \frac{V_0 \cdot V_i}{\rho_0}\,dx \geq 0,
$$ 
then we keep $ V_i $ as it is;  otherwise, if  
$$ 
\int_Q \frac{V_0 \cdot V_i}{\rho_0} \,dx< 0,
$$ 
then we replace $ V_i $ by $ -V_i $. 

With this choice, it follows that there exists a sequence $\{V_i\}$ such that 
\begin{equation}
\label{condition2 of V and Vi}
\left| \int_Q \frac{V_0 \cdot (\widetilde{V} + V_i)}{\rho_0}\,dx \right| \leq \frac{1}{8} \int_Q \frac{|\widetilde{V} + V_i|^2}{\rho_0}\,dx.
\end{equation}
Thus,  by combining  \eqref{condition1 of V+Vi} and \eqref{condition2 of V and Vi}, we conclude that 
$V_i + \widetilde{V} \in X_0$.  

\textbf{Step 3.}
Note that the uniform continuity of the functions $\rho$, $V_0$, $R_0$, and $(\widetilde{V}, \widetilde{U})$ in $\Omega$ guarantees that the size of the cube $Q$ employed in Step~1 and Step~2 can be chosen independently of the point $(x_0, t_0) \in \Omega_1$. 

Fixing a sufficiently small $\delta < \delta_1$, we construct a grid in $\mathbb{R}^n \times \mathbb{R}^+$ with nodes at
\[
(m_1 \delta, \dots, m_n \delta, l \delta), \;\text{where}\;\; l \in \mathbb{N}, \; (m_1, \dots, m_n) \in \mathbb{Z}^n .
\]
From this grid, we select a finite collection of cubes of size $\delta$, with vertices lying on the grid, that cover $\Omega_1$. Denoting these cubes by $\{Q_k : k = 1, \dots, N\}$, we obtain a family with pairwise disjoint interiors, all contained within $\Omega$.

For each $k$, let $\{(V_i^k, U_i^k)\}_{i \in \mathbb{N}}$ denote the sequence of functions, compactly supported in $Q_k$, constructed as in Step~1 and Step~2. We then define the global perturbations in the domain $P$ by
\[
 \widetilde{V}_i :=  \widetilde{V} + \sum_{k=1}^N V_i^k, \qquad  \widetilde{U}_i :=  \widetilde{U} + \sum_{k=1}^N U_i^k.
\]

For a fixed $ i $, the pairs $ (V_i^k, U_i^k) $ for $ k = 1, \dots, N $ have disjoint supports. Hence, by Step 1 and Step 2, we conclude that $ \widetilde{V}_i \in X_0 $. Moreover, for each fixed $ k $, $ V_i^k \rightharpoonup 0 $ weakly, which implies that $ \widetilde{V}_i \rightharpoonup 0 $ weakly in the limit as $ i \to \infty $. Finally, using 
\eqref{punch3}, we obtain
\begin{equation*}
\begin{split}
 \quad \left\| \widetilde{V}_i - \widetilde{V} \right\|_{L^1(P)}& \geq \sum_{k=1}^N \|V_i^k\|_{L^1(Q_k)} \geq \frac{c_1}{4\Lambda} \int_{Q_k} \trace  M \, dx \, dt 
\\&\geq \frac{c_1}{4\Lambda} \int_{\Omega_1} \trace  M \, dx \, dt \geq \frac{c_1}{8\Lambda} \int_P \trace  M \, dx \, dt.
\end{split}
\end{equation*}
This concludes the proof of the lemma.
\end{proof}

This lemma yields the following proposition, which provides control over oscillations via the Reynolds stress term.

\begin{Proposition}
\label{proposition points}
Let $(\widehat{V},\widehat{U})\in X$ be a point of continuity of the identity map $\id$ from $(X,d)$ to $L^2(\R^n\times\R)$. Then, $(\widehat{V},\widehat{U})$ satisfies \eqref{restriction =}. In addition, there exists a positive number $C_0>0$, such that 
 $$
 \int_0^T\int_{\mathbb{T}^n}\frac{|\widehat{V}|^2}{\rho_0}\,dx\,dt\leq C_0\int_0^T\int_{\mathbb{T}^n} \trace R_0\,dx\,dt.
 $$
\end{Proposition}
\begin{proof}
By definition, there exists a sequence $\{(V_j, U_j)\} \subset X_0$ that converges weak-* to $(\widehat{V}, \widehat{U})$. In particular, we have
\[
V_j \to \widehat{V} \quad \text{strongly in } L^2(P),
\]
and hence also strongly in $L^1_{\text{loc}}(P)$.

From Lemma \ref{L1-coercivity}, for each $(V_j, U_j)$ there exists a subsequence $\{(V_{j,i}, U_{j,i})\}_{i\in \mathbb{N}}$ converging weak-* to $(V_j, U_j)$. Note that $(V_j, U_j)$ satisfies \eqref{ci system}, and we have the estimate
\[
\|V_{j,i} - V_j\|_{L^1(P)} \geq \frac{c_1}{8\Lambda} \int_P \trace M_j \, dx\,dt,
\]
where $M_j$ is given by
\[
M_j := \frac{|V_0|^2}{n\rho_0} \id_n + R_0 - \frac{(V_0 + V_j) \otimes (V_0 + V_j)}{\rho_0} + U_0 + U_j.
\]

By a diagonal argument, we may extract a subsequence $\{(V_{j,i(j)}, U_{j,i(j)})\}$ that converges weakly-* to $(\widehat{V}, \widehat{U})$ and satisfies
\[
\lim_{j \to \infty} \|V_{j,i(j)} - \widehat{V}\|_{L^1(P)} \geq \frac{c_1}{8\Lambda} \int_P \trace M_j \, dx\,dt.
\]
  This yields to 
\begin{equation}
  \label{limit of Mj}
  \lim_{j}\int_{P}\trace M_j\,dx\,dt=0.
\end{equation}
  
We further introduce  
\begin{equation*}
  \widehat{M} := \frac{|V_0|^2}{n\rho_0} \id_n +R_0-\frac{(V_0+\widehat{V})\otimes (V+\widehat{V})}{\rho_0}+U_0+\widehat{U}.
  \end{equation*}
It follows that $M_j\to \widehat{M}$ weak-$*$. Note that $M_j>0$, which implies that $\widehat{M}\geq 0$ a.e.. From \eqref{limit of Mj} and noting that 
$$
\trace M_j\to \trace \widehat{M}\;\;\text{weak-}*,
$$
one obtains  
\begin{equation*}
  \lim_{j}\int_{P} \trace M_j\,dx\,dt = \int_P\tr \widehat{M}\,dx\,dt=0.
\end{equation*}
   Thus, $\tr \widehat{M}=0$ and so $\widehat{M}=0$, which yields that $(\widehat{V},\widehat{U})$ satisfies  \eqref{restriction =}.
    
Meanwhile, note from the definition of $\widehat{M}$ that 
  \begin{equation*}
  \begin{split}
  \tr \widehat{M}&=\frac{|V_0|^2}{\rho_0} + \tr R_0-\frac{|V_0+\widehat{V}|^2}{\rho_0}
  = \tr R_0-\frac{|\widehat{V}|^2}{\rho_0}-\frac{2V_0\cdot \widehat{V}}{\rho_0}.
  \end{split}
  \end{equation*}  
  This yields that
  \begin{equation*}
  \begin{split}
  \int_P\tr R_0\,dx&=\int_P\frac{|\widehat{V}|^2}{\rho_0}\,dx+2\int_P\frac{V_0\cdot\widehat{V}}{\rho_0}\,dx
  \\&\geq \int_P\frac{|\widehat{V}|^2}{\rho_0}\,dx-\frac{1}{4}\int_P\frac{|\widehat{V}|^2}{\rho_0}\,dx
  \\&=\frac{3}{4}\int_P\frac{|\widehat{V}|^2}{\rho_0}\,dx,
  \end{split}
  \end{equation*}
completing the proof.   
\end{proof}

Using Proposition \ref{proposition points}, a Bair\'e category argument leads to the following lemma.
\begin{Lemma}\label{lem conv int}
Let $(\rho_0, V_0, R_0) \in C^0(\R \times \R^n; \R \times \R^n \times \mathbb{S}_0^{n\times n})$
 be given with $\rho_0$ satisfying  \eqref{bounds rho}  and $R_0$ being positive definite in some open set $P \subset \R \times \R^n$. Let $U_0$ be given as in \eqref{U0 condition}. There exist infinitely many $(\widetilde{V},\widetilde{U}) \in L^\infty(\R \times \R^n; \R^n \times  \mathbb{S}_0^{n\times n})$ which are compactly supported in $P$ and satisfy \eqref{ci system} and \eqref{restriction =}. In addition, there exists a positive number $C_0>0$, such that 
 $$\int_0^T\int_{\mathbb{T}^n}\frac{|\widetilde{V}|^2}{\rho_0}\,dx\,dt\leq C_0\int_0^T\int_{\mathbb{T}^n} \tr R_0\,dx\,dt.$$
\end{Lemma}

\bigskip

\section{The proof of the Main result}\label{sec pf}

To prove our main theorem, we will rely on the following lemma.
\begin{Lemma}
\label{global existence}
For any $\eps > 0$, there exists a $\delta_0 = \delta_0(\eps) > 0$ such that for any $0 < \delta \le \delta_0$, there exist infinitely many weak solutions $(\rho_{\delta},\rho_{\delta}\widehat{\v}_{\delta})$ to \eqref{scalling system 111}on $[0, T- \eps]$, satisfying  
$$
\|\widehat{\v}_{\delta} - \v_{\delta}\|_{L^2([0,T - \eps] \times \mathbb{T}^n)}^2\leq C \varepsilon,
$$
where $(\rho_\delta, \rho_\delta \v_\delta, R = R^\varepsilon + \tilde{R}_\delta)$ is the subsolution solving \eqref{eq subsoln}. 
\end{Lemma}
\begin{proof}
Recalling Lemma \ref{subsolution}, we have
\[
\rho_{\delta} \v_{\delta} = \m_{\delta} + \u^{\varepsilon}, \quad \text{and} \quad \rho_\delta(t,x) = 1 + \delta^2 \varrho_{\delta, K_\ast}(t,x), 
\]
where $\varrho_{\delta, K_\ast}$ is given in \eqref{eq varrho}. From Lemma \ref{lem conv int} it follows that there exist infinitely many weak solutions to \eqref{scalling system 111} of the form $({\rho}_{\delta}, {\rho}_{\delta} \widehat{\v}_{\delta})$.

From the convex integration procedure in the previous section, we construct the perturbation as
\[
\widetilde{V} = \rho_\delta \widehat{\v}_{\delta} - \rho_\delta \v_{\delta}.
\]
So applying Lemma~\ref{lem conv int} again, we have the estimate
\begin{align*}
\| \widehat{\v}_{\delta} - \v_{\delta} \|^2_{L^2([0,T - \eps] \times \mathbb{T}^n)} &= \int_0^{T- \eps}\int_{\mathbb{T}^n} \frac{|\widetilde{V}|^2}{\rho_\delta^2} \, dx\,dt \\
&\leq c \int_0^{T- \eps}\int_{\mathbb{T}^n} \frac{|\widetilde{V}|^2}{\rho_\delta} \, dx\,dt \\
& \leq C \int_0^{T - \eps}\int_{\mathbb{T}^n} \trace R \, dx\,dt,
\end{align*}
where $ R = R^{\varepsilon} + \tilde{R}_{\delta}$, and we have used the estimate on $\rho_\delta$ in \eqref{eq subsoln est} to derive the first inequality. From the definition of $R^\eps$ we have
\[
\tr R^\eps = \tr (\u \otimes \u)^\eps - \tr (\u^\eps \otimes \u^\eps) + \frac{n \eps}{2} = (|\u|^2)^\eps - |\u^\eps|^2 + \frac{n \eps}{2} > 0,
\]
and thus by Lemma \ref{incom smooth},
\[
\int_0^{T - \eps}\int_{\mathbb{T}^n} \trace R^\eps \, dx\,dt \le \|(\u\otimes \u)^{\varepsilon}-\u^\eps \otimes \u^\eps\|_{L^1([0,T -\eps] \times \mathbb{T}^n)} + \int_0^{T - \eps}\int_{\mathbb{T}^n} \frac{n \eps}{2} \, dx\,dt \le C \eps. 
\]
Recall from Lemma \ref{subsolution} that for $\delta < \delta_0$,  $ \| \tilde{R}_\delta \|_{{L^\infty([0,T - \eps] \times \mathbb{T}^n)}} \leq C \varepsilon $. Together we obtain
\[
\| \widehat{\v}_{\delta} - \v_{\delta} \|^2_{L^{2}([0,T - \eps] \times \mathbb{T}^n)} \leq C \varepsilon,
\]
completing the proof of the lemma. 
\end{proof}

Equipped with the preceding lemmas, we are now prepared to establish our main result. 
\begin{proof}[Proof of Theorem \ref{main result}]
Given any $L^2$-bounded weak solution $\u$ to the incompressible Euler equations \eqref{incom Euler} on the time interval $[0, T)$, for any $0 < T' < T$ there exists $\eps > 0$ with $T' < T - \eps$. We apply Lemma \ref{smooth euler}, Lemma \ref{subsolution}, and Lemma \ref{global existence} to construct a family of weak solutions $(\rho_{\delta},\rho_{\delta}\widehat{\v}_{\delta})$ to the scaled compressible Euler equations \eqref{scalling system 111} on the time interval $[0, T-\eps]$, and hence on $[0, T']$. Note that here $\delta = \delta(\eps)$ with the property that $\delta \to 0 \ \Leftrightarrow \ \eps \to 0$. 
We can take the following initial data  
$$
(\rho_{\delta}^0, \m_\delta^0)(x) := (\rho_{\delta}, \rho_{\delta}\widehat{\v}_{\delta})(0,x).
$$

Lemma \ref{subsolution} ensures that
\[
\rho_{\delta}\to1 \quad \text{in } C^0([0,T'] \times\mathbb{T}^n) \qquad \text{as } \ \delta \to 0.
\]
From Lemma \ref{global existence} and Lemma \ref{subsolution} we find that 
\begin{align*}
\| \rho_\delta \widehat{\v}_{\delta} - \rho_\delta \v_{\delta} \|_{L^{2}([0,T'] \times \mathbb{T}^n)} & = \| \widetilde{V} \|_{L^{2}([0,T'] \times \mathbb{T}^n)} \\
& \le C \| \sqrt{\rho_\delta} \|_{L^\infty([0,T'] \times \mathbb{T}^n)} \sqrt{\| \trace R \|_{L^{1}([0,T'] \times \mathbb{T}^n)}} \\
& \le C \eps + C(\eps) \delta \ \ \to 0 \qquad \text{as } \ \delta \to 0, \\
\| \sqrt{\rho_\delta} \widehat{\v}_{\delta} - \sqrt{\rho_\delta} \v_{\delta} \|_{L^{2}([0,T'] \times \mathbb{T}^n)} & = \| \widetilde{V}/ \sqrt{\rho_\delta} \|_{L^{2}([0,T'] \times \mathbb{T}^n)} \\
& \le C \sqrt{\| \trace R \|_{L^{1}([0,T'] \times \mathbb{T}^n)}} \\
&  \le C \eps + C(\eps) \delta \ \ \to 0 \qquad \text{as } \ \delta \to 0.
\end{align*}
Further evoking \eqref{bound of density in delta}, \eqref{m close} and \eqref{eq reg est} we see that 
\begin{align*}
\| \rho_\delta \widehat{\v}_{\delta} - \u \|_{L^{2}([0,T'] \times \mathbb{T}^n)} & \le \| \rho_\delta \widehat{\v}_{\delta} - \rho_\delta \v_{\delta} \|_{L^{2}([0,T'] \times \mathbb{T}^n)} + \| \rho_\delta \v_{\delta} - \u^\eps \|_{L^{2}([0,T'] \times \mathbb{T}^n)} \\
&\qquad + \| \u^\eps - \u \|_{L^{2}([0,T'] \times \mathbb{T}^n)} \\
& \le C \eps + C(\eps) \delta \ \ \to 0 \qquad \text{as } \ \delta \to 0,\\
\| \sqrt{\rho_\delta} \widehat{\v}_{\delta} - \u \|_{L^{2}([0,T'] \times \mathbb{T}^n)} & \le \| \sqrt{\rho_\delta} \widehat{\v}_{\delta} - \rho_\delta \widehat{\v}_{\delta} \|_{L^{2}([0,T'] \times \mathbb{T}^n)} + \| \rho_\delta \widehat{\v}_{\delta} - \u\|_{L^{2}([0,T'] \times \mathbb{T}^n)}\\
& \le \| \sqrt{\rho_\delta} - \rho_\delta \|_{L^{\infty}([0,T'] \times \mathbb{T}^n)} \| \widehat{\v}_\delta \|_{L^{2}([0,T'] \times \mathbb{T}^n)} + \| \rho_\delta \widehat{\v}_{\delta} - \u\|_{L^{2}([0,T'] \times \mathbb{T}^n)}\\
& \le C \eps + C(\eps) \delta \ \ \to 0 \qquad \text{as } \ \delta \to 0.
\end{align*}

For any $\phi \in C^\infty_c([0, T) \times \mathbb{T}^n; \R^n)$ with $\Dv\phi=0$, say the time support of $\phi$ is in $[0, T']$ for some $T' < T$. Then
\begin{equation*}
\begin{split}
\int^{T'}_0 \int_{\mathbb{T}^3} \rho_{\delta}\widehat{\v}_{\delta} \cdot \partial_t \phi + \sqrt{\rho_{\delta}}\widehat{\v}_{\delta} \otimes\sqrt{\rho_{\delta}}\widehat{\v}_{\delta} : \nabla \phi  \,dxdt +\int_0^{T'} \int_{\mathbb{T}^3}\frac{\rho_{\delta}^{\gamma}}{\delta^2}\Dv \phi\,dx\,dt
=  -\int_{\mathbb{T}^3} \m_{\delta}^0 \cdot \phi(\cdot, 0) \,dx,
\end{split}
\end{equation*}
Using the convergence results in the above, we find that the left-hand side of the above equality converges, as $\delta \to 0$, to
$$
\int^{T'}_0 \int_{\mathbb{T}^3} \u \cdot \partial_t \phi + \u \otimes \u : \nabla \phi  \,dxdt, 
$$
which, recalling the fact that $\u$ is a weak solution to the incompressible Euler equation, is equal to
\[
-\int_{\mathbb{T}^3} \u^0 \cdot \phi(\cdot, 0) \,dx.
\]
Therefore, we must have 
$$
\int_{\mathbb{T}^3} \m_{\delta}^0 \cdot \phi(\cdot, 0) \,dx\to \int_{\mathbb{T}^3} \u^0 \cdot \phi(\cdot, 0) \,dx \qquad \text{as }\ \delta \to 0.
$$
Putting together all of the above, we have proved the main result, Theorem \ref{main result}.
\end{proof}


\bigskip

\section*{Acknowledgments}
Robin Ming Chen is partially supported by the NSF grant DMS-2205910. 
Alexis Vasseur is partially supported by the NSF grants DMS-2219434 and DMS-2306852. 
Dehua Wang is partially supported by the NSF grant DMS-2510532. 
Cheng Yu is partially supported by the NSF grant DMS-2510425 and by the Simons Foundation MPS-TSM-00007824.


\bigskip\bigskip
\begin{thebibliography}{99}
\bibitem{BB}
J. Bourgain and H. Brezis, On the equation $\Dv Y = f$ and application to control of phases,  J. Amer. Math. Soc., 16 (2003) no. 2, 393--426.

 \bibitem{BDGL}D. Bresch, B. Desjardins, E. Grenier, and C.-K. Lin, Low Mach number limit of viscous polytropic flows: formal
asymptotics in the periodic case, Stud. Appl. Math., 109 (2): 125--149, 2002.


\bibitem{BFH}D. Breit, E. Feireisl, M. Hofmanova, Incompressible limit for compressible fluids with stochastic forcing,
Arch. Ration. Mech. Anal. 222 (2016), no. 2, 895--926.

\bibitem{Annals} T. Buckmaster, V. Vicol, Nonuniqueness of weak solutions to the Navier--Stokes equation, Ann. of Math. (2) 189 (2019), no. 1, 101--144.

\bibitem{CVY} R. M. Chen, A. Vasseur,  C. Yu,   Global ill-posedness for a dense set of initial data to the isentropic system of gas dynamics, Adv. Math. 393 (2021), Paper No. 108057, 46 pp.

\bibitem{CJ} Y.-P. Choi, J. Jung, Incompressible Navier--Stokes limit from nonlinear Vlasov-Fokker-Planck equation, Appl. Math. Lett. 158 (2024), Paper No. 109214, 7 pp.


\bibitem{CDK} E. Chiodaroli, C. De Lellis, O. Kreml,
Global ill-posedness of the isentropic system of gas dynamics, Comm. Pure Appl. Math. 68 (2015), no. 7, 1157-1190.



\bibitem{CF} E. Chiodaroli, E. Feireisl,
On the density of ``wild" initial data for the barotropic Euler system, Ann. Mat. Pura Appl. (4)203(2024), no.4, 1809--1817.

\bibitem{DS}
C. De Lellis, L. Szekelyhidi Jr., The Euler equations as a differential inclusion, Ann. Math. (2009) 1417--1436.


\bibitem{Ebin} D. G. Ebin, The motion of slightly compressible fluids viewed as a motion with strong constraining force, Ann. of Math. (2), 105(1):141--200, 1977.



\bibitem{Feireisl-incom} 
E. Feireisl, Incompressible limits and propagation of acoustic waves in large domains with boundaries, Comm. Math. Phys., 294(1):73--95, 2010.

\bibitem{FKM}
E. Feireisl, C. Klingenberg, S. Markfelder, On the low Mach number limit for the compressible Euler
system, SIAM J. Math. Anal. 51(2):1496–1513, 2019

\bibitem{FKMV} E. Feireisl, O. Kreml, V. Macha, S.Necasova, 
On the low Mach number limit of compressible flows in exterior moving domains, J. Evol. Equ. 16 (2016), no. 3, 705--722.


\bibitem{Fu}
M. Fujii,  Low Mach number limit of the global solution to the compressible Navier--Stokes system for large data in the critical Besov space, Math. Ann. 388 (2024), no. 4, 4083--4134.


\bibitem{KM}
S. Klainerman and A. Majda, Singular limits of quasilinear systems with large
parameter and the incompressible limit of compressible fluids, Comm. Pure Appl. Math. 34 (1981), 481-524.


\bibitem{KM2} S. Klainerman and A. Majda, Compressible and incompressible fluids, Comm.
Pure Appl. Math. 35 (1982), 637-656.


\bibitem{LM} P.-L. Lions and N. Masmoudi, Incompressible limit for a viscous compressible fluid, J. Math. Pures Appl. (9),
77(6):585--627, 1998.

\bibitem{LM2} P.-L. Lions and Nader Masmoudi, Une approche locale de la limite incompressible, C. R. Acad. Sci. Paris Ser. I Math., 329(5):387--392, 1999.


\bibitem{LM3} P.-L. Lions, N. Masmoudi, Incompressible limit for a viscous compressible fluid, J. Math. Pures Appl. (9) 77 (1998), no. 6, 588--627. 

\bibitem{Majda}
A. Majda, Compressible fluid flow and systems of conservation laws in several space variables, Applied Mathematical Sciences. Springer-Verlag New York, 1984.

\bibitem{MaS}
A. Majda, J. A. Sethian, The derivation and numerical solution of the equations for zero Mach number combustion, Combustion Science and Technology, 42(3-4), 185--205.


\bibitem{Scho}S. Schochet, The compressible Euler equations in a bounded domain: existence of solutions and the incompressible limit, Comm. Math. Phys., 104(1): 49--75, 1986.


\bibitem{MS2} G. Metivier and S. Schochet, Averaging theorems for conservative systems and the weakly compressible Euler equations, J. Differential Equations, 187(1): 106--183, 2003.

\bibitem{MS0} G. Metivier and S. Schochet, The incompressible limit of the non-isentropic Euler equations, Arch. Ration. Mech. Anal., 158(1): 61--90, 2001.

\bibitem{MK} S. Markfelder and C. Klingenberg, The Riemann problem for the multidimensional isentropic system of gas dynamics is ill-posed if it contains a shock, Arch. Ration. Mech. Anal., 227:967--994, 2018.


\bibitem{SEC}P. Secchi, The incompressible limit of the equations of compressible ideal magneto-hydrodynamics with perfectly conducting boundary, Commun. Math. Anal. Appl. 3 (2024), no. 2, 168--198.

\bibitem{SW}
{L.~Sz\'{e}kelyhidi, E.~Wiedemann}, { Young measures generated by
  ideal incompressible fluid flows}, Arch. Rational. Mech. Anal., 206 (2012),
  ~333--366.

\bibitem{Ukai}S. Ukai, The incompressible limit and the initial layer of the compressible Euler equation, J. Math. Kyoto Univ., 26(2): 323--331, 1986.



\end{thebibliography}
\end{document}